\documentclass[12pt]{article}%
\usepackage{amsfonts}
\usepackage{amsmath}
\usepackage{amssymb}
\usepackage{amsxtra}
\usepackage{graphicx}
\usepackage{geometry}
\usepackage{color}
\usepackage{caption}%
\usepackage{enumerate}
\providecommand{\U}[1]{\protect\rule{.1in}{.1in}}
\newtheorem{theorem}{Theorem}

\newtheorem{proposition}[theorem]{Proposition}

\newenvironment{proof}[1][Proof]{\noindent\textbf{#1.} }{\ \rule{0.5em}{0.5em}}
\numberwithin{equation}{section}

\numberwithin{equation}{section}
\numberwithin{figure}{section}

\def\bfa{{\bf a}}

\def\bfc{{\bf c}}

\def\bfu{{\bf u}}
\def\bfv{{\bf v}}
\def\bfw{{\bf w}}

\def\bfA{{\bf A}}
\def\bfB{{\bf B}}

\def\bfH{{\bf H}}

\def\bfQ{{\bf Q}}

\def\bfU{{\bf U}}
\def\bfV{{\bf V}}
\def\bfW{{\bf W}}
\def\bfX{{\bf X}}

\def\bfLa{\mbox{\boldmath$\Lambda$}}
\def\bfla{\mbox{\boldmath$\lambda$}}
\def\bfmu{\mbox{\boldmath$\mu$}}
\def\bfve{\mbox{\boldmath$\varepsilon$}}

\def\bbN{{\mathbb N}}

\def\bbR{{\mathbb R}}

\def\bbZ{{\mathbb Z}}

\def\cD{{\mathcal D}}

\def\cH{{\mathcal H}}

\def\cK{{\mathcal K}}

\def\cT{{\mathcal T}}

\def\cW{{\mathcal W}}

\newcommand{\ve}{\varepsilon}

\newcommand{\dist}{\mathop{\rm dist}\nolimits}

\renewcommand{\div}{\mathop{\rm div}\nolimits}

\begin{document}

\title{Effects of Rayleigh waves on the essential spectrum in perturbed 
doubly periodic elliptic problems}
\author{F.Bakharev\thanks{Chebyshev Laboratory, St. Petersburg State University, 14
th Line, 29b, Saint Petersburg, 199178 Russia; email: fbakharev@yandex.ru},
G.Cardone\thanks{Universit\`{a} del Sannio, Department of Engineering, Corso
Garibaldi, 107, 82100 Benevento, Italy; email: giuseppe.cardone@unisannio.it},
S.A.Nazarov\thanks{St. Petersburg State University, 198504, Universitetsky pr., 28, Stary Peterhof, Russia; Peter the Great St. Petersburg State Polytechnical University, Polytechnicheskaya ul., 29, St. Petersburg, 195251, Russia; Institute of Problems of Mechanical Engineering RAS, V.O., Bolshoj pr., 61, St. Petersburg, 199178, Russia; email:
srgnazarov@yahoo.co.uk.},
J.Taskinen\thanks{University of Helsinki, Department of Mathematics and Statistics,
P.O. Box 68, 00014 Helsinki, Finland; email: jari.taskinen@helsinki.fi.}}
\maketitle

\begin{abstract}
We give an example of a scalar second order differential operator in the plane with double periodic 
coefficients and describe its modification, which causes an additional spectral band in the essential 
spectrum. The modified operator 
is obtained by applying to the coefficients a mirror reflection with respect to a vertical or 
horizontal line. This change gives rise to Rayleigh type waves localized near the line. The results 
are proven using asymptotic analysis, and they are based on high contrast of the coefficient 
functions.
\medskip

Keywords: periodic media, open waveguides, high contrast of coefficients, asymptotics, spectral bands.

\medskip

MSC: Primary 35P05; Secondary 47A75.

\end{abstract}






\section{Introduction}
\label{sec1}

\subsection{Motivation.}
\label{sec1.1}

A satisfactory theory for spectral
elliptic boundary-value problems in  double periodic media containing open waveguides
does not exist yet, and the topic contains a lot of unanswered questions. An open waveguide
consists of a semi-infinite foreign inclusion, cf. Fig.\,\ref{fig1}, and being a non-compact
domain perturbation, it can in general change the essential spectrum of the problem, when 
compared to the corresponding problem on an intact  domain without perturbation. This topic was studied for example in the recent  paper \cite{CaNaTa}, which contains a complete 
description of the 
essential spectrum $\sigma_{\rm ess}(\cT)$ for a large class of elliptic second order systems
with   Neumann boundary conditions, satisfying a Korn inequality. 
The following question\footnote{The question has been asked by a referee of the paper  \cite{CaNaTa}, among others} has arisen in the course of the investigation:
is the formula
\begin{equation}
\label{5}
\sigma_{\rm ess}(\mathcal{T})=\sigma_{\rm ess}^-\cup\sigma_{\rm ess}^+
\end{equation}
valid for an elliptic problem in the union of 
two subdomains of the plane, which are contained in the lower and upper half-planes;
here, $\sigma_{\rm ess}^-$ (respectively, $\sigma_{\rm ess}^+$) 
denotes the essential spectrum of the corresponding problem in the lower (resp. upper) half-plane,
and it is assumed that these two problems are  periodic  along the abscissa axis, but
independently of each other?
Obviously, the interesting aspect in this problem is to find a possible component of 
$\sigma_{\rm ess} (\mathcal{T})$ which is not contained in the spectra of the problems in the subdomains.

\begin{figure}[t]
\begin{center}
\includegraphics[scale=0.4]{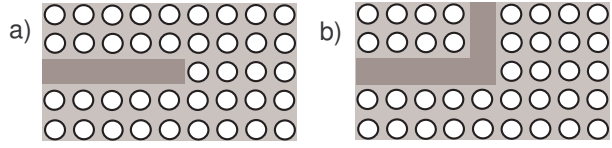}
\end{center}
\caption{Semi-infinite (a) and angular (b) open wave\-guides in double-periodic planar domains.}
\label{fig1}
\end{figure}

As additional motivation of the problem we recall the case of the one-dimensional
Schr\"odinger equation
\begin{equation}
\label{1}
-\partial^2_x w(x)+V(x) w(x)=\lambda w(x), \quad x\in \bbR=(-\infty,+\infty)
\end{equation}
with the composite potential
\begin{equation}
\label{2}
V(x)=V^{\pm}(x)\quad \mbox{for} \quad \pm x>\ell>0 ,
\end{equation}
where $\partial_x=\partial/{\partial x}$ and $V^\pm$ are 1-periodic positive smooth
functions; smoothness is assumed here for the sake of simplicity. The essential spectrum
$\sigma_{\rm ess}$ of the problem \eqref{1} is just the union of the spectra $\sigma_{\rm ess}^\pm$ of the
differential operators $-\partial_x^2+V^\pm$ with periodic coefficients in the whole axis $\bbR$. This fact is evident
because the equation can be reformulated as s system of ordinary differential equations
\begin{equation}
\label{3}
-\partial_x^2w^\pm(x)+V(x)w^\pm(x)=\lambda w^\pm (x), \quad x\in\bbR_\pm
=\{x\in\bbR : \pm x>0\},
\end{equation}
with  transmission conditions
\begin{equation}
\label{4}
w^+(+0)=w^-(-0),\quad \partial_x w^+(+0)=\partial_x w^{-}(-0)\,.
\end{equation}
Indeed, according to \eqref{2}, the essential spectrum of \eqref{3} with  Dirichlet conditions $w^\pm(0)=0$
is nothing but $\sigma_{\rm ess}^\pm$, while the system \eqref{3},
\eqref{4} differs from the couple of the Dirichlet problems in $\bbR_\pm$
by a localized perturbation (it can be interpreted as a compact perturbation).

\begin{figure}[t]
\begin{center}
\includegraphics[scale=0.4]{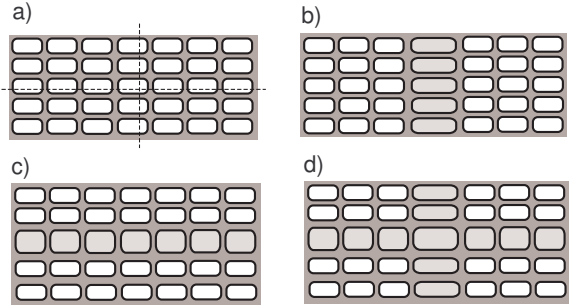}
\end{center}
\caption{Double-periodic planar domain (a) and composite domains (b-d) created by mirror
reflections.}
\label{fig2}
\end{figure}

Coming back to the problem \eqref{5}, the above described argument based on  a compact perturbation works no longer, since the interface 
$\partial \bbR^2_\pm$, where $\bbR^2_\pm = \{ (x_1,x_2) \in \bbR^2 \, :\, \pm  x_1 \geq 0 \} $, is infinite. 
However, there is no easy,
satisfactory counterexample to the relationship \eqref{5}, if the natural
requirements like  
smoothness of the coefficient are to  be satisfied; this prevents answering the question
directly by using classical Rayleigh waves \cite{Ray}, \cite{SO} in elasticity and their 
generalizations, see \cite{KamKisel}.

In the present paper we give examples of  elliptic
scalar equations with smooth double periodic coefficients, which have the following
property:
if the plane is divided along dotted lines in Fig.\,\ref{fig2}.a) and
the left  or upper  half-plane is doubled by  using  mirror reflection,
Fig.\,\ref{fig2}.b) or c), the new  elliptic problem 
gains  essential spectrum $\sigma_{\rm ess}(\cT)$ with  at least one additional spectral band
in comparison with the original spectrum $\sigma_{\rm ess}(\cT^0)$ of the
double periodic problem. This main result of our paper is formulated in Theorem \ref{thmain} 
in Section \ref{sec3}. We also mention that the result of \cite{CaNaTa}
is proven in two steps: the first consists of finding a
singular Weyl sequence at any point $\lambda\in\sigma_{\rm ess}(\cT)$ for the problem operator
$\cT$ and the second of   the construction of a (right) parametrix for the problem
with any  $\lambda \notin \sigma_{\rm ess}(\cT)$. It has been asked, if it is possible to 
avoid the quite technical and cumbersome construction of the parametrix, 
like it has been done in the case of the one-dimensional Schr\"odinger equation.
The present paper also demonstrates the complications in this respect.

To fulfill the task, we employ an  elegant formulation  of \cite{Hempel},
see also \cite{HeO,Zhik, na461}, on the  detection of spectral gaps in scalar
problems, where the  coefficients of the differential operator have high contrast.
However, we were not able to apply these results directly, and modifications 
are presented in  Section \ref{sec3} in order to satisfy all natural 
assumptions. In particular, we find a way to keep the 
infinite smoothness of the  coefficients; note that in \cite{Hempel, HeO, Zhik, na461}
the coefficients have to be piecewise constant. In particular, as is drafted in
Fig.\,\ref{fig2}.a), the massive  hard parts of the double  periodic medium are separated
by thin, soft "mortar"
like in  hand-made masonry (similar structure appears also in natural quarzites).
Compared with the   citations, especially \cite{na461} where a similar geometric
structure was employed for a different purpose, we use 
quite a different scheme of asymptotic analysis, which also leads  to new   asymptotic
results about the purely periodic case in Section \ref{sec2} (Theorem \ref{th1}). In order
to clarify the proof  of our main result we will accept some simplifying assumptions. 
Possible generalizations will be  discussed in Sect. \ref{sec4}.

\subsection{Purely periodic medium.}
\label{sec1.2}
We now describe the double  periodic elliptic second order partial differential equation
which will be investigated in Section \ref{sec2}. The main 
example of the failure
of the equality \eqref{5} for the composite medium $\bbR_+^2\cup \bbR_-^2$ will be constructed in Section \ref{sec3}.

We define the period cell as the rectangle $Q=(-\ell_1,\ell_1)\times (-\ell_2,\ell_2)$
with $\ell_1\geq \ell_2>0$. For $\ve\in (0,\ell_2)$ we introduce a smaller rectangle
$Q_{\ve}=(-\ell_1+\ve,\ell_1-\ve)\times (-\ell_2+\ve,\ell_2-\ve)$.
Let us also define a family of translated domains
$$
Q_{\ve}(\alpha)=\{x=(x_1,x_2): (x_1-2\ell_1\alpha_1,x_2-2\ell_2\alpha_2)\in Q_{\ve}\}
$$
where $\alpha=(\alpha_1,\alpha_2)\in \bbZ^2$ and $\bbZ=\{0,\pm1,\pm2,\ldots\}$.

We consider the spectral problem
\begin{equation}
\label{6}
-\div(a^\ve(x) \nabla_x u^\ve (x))=\lambda^\ve u^\ve(x),\quad x\in \bbR^2,
\end{equation}
where $\nabla_x$ is the gradient in the variable $x$  and $\lambda^\ve$ is a spectral parameter. The function $a^\ve$ is smooth and
$2\ell_j$-periodic in $x_j$ such that
\begin{equation}
\label{7}
a^\ve(x)=1, \quad  x\in Q_{2\ve},\quad a^\ve(x)=\ve^{2\gamma}, \quad x\in Q\setminus \overline{Q_\ve},
\end{equation}
and $a^\ve(x)\in (\ve^{2\gamma},1]$ if $x\in Q_{\ve}\setminus Q_{2\ve}$, where
 $\gamma \in (1/2,1)$ is a fixed parameter.

The variational formulation of the problem \eqref{6} reads as
\begin{equation}
\label{8}
(a^\ve \nabla_x u^\ve, \nabla_x v)_{\bbR^2}=\lambda^\ve (u^\ve,v)_{\bbR^2}, \quad v\in H^1(\bbR^2)
\end{equation}
where $(f,g)_\Omega$ stands for the usual (complex valued) inner product in $L^2(\Omega)$
for  a domain   $\Omega \subset \bbR^2$.
We denote the standard Sobolev space by $H^1(\Omega)$. The sesquilinear form
on the left of \eqref{8} is positive and closed in $H^1(\bbR^2)$ and consequently (see \cite[Ch. 10]{BiSo}, \cite[Thm. VIII.5]{RS1})
our problem can be rewritten as an abstract operator equation  $\cT^0(\ve) u^\ve=\lambda^\ve u^\ve$,
where $\cT^0(\ve)$ is an unbounded positive self-adjoint operator in Hilbert space $L^2(\bbR^2)$ with
domain $\cD(\cT^0(\ve))=H^2(\bbR^2)$,
and thus the spectrum $\sigma(\cT^0(\ve))$ is a subset of the semi-axis $\overline{\bbR_+}=[0,+\infty)$.
The embedding $H^1(\bbR^2)\subset L^2(\bbR^2)$ is not compact, hence the essential
spectrum $\sigma_{\rm ess}(\cT^0(\ve))$ is not empty.

\section{Asymptotic analysis of the spectrum of the purely periodic problem}
\label{sec2}

\subsection{FBG-transform and  model problem in the period cell.}
\label{sec2.1}
The Floquet-Bloch-Gelfand-(FBG-)transform, see \cite{Gel} and also \cite{DS,KuchUMN,Skrig,Kuchbook}, converts the differential equation \eqref{6} into the following  problem
with quasiperiodic boundary conditions in the period cell $Q$,
\begin{eqnarray}
\label{15}
&&-\div (a^\ve(x)\nabla_x U^\ve(x;\eta))=\Lambda^\ve(\eta)U^\ve(x;\eta),\quad x\in Q,
\\
\label{16}
&&U^\ve(x;\eta)|_{x_j=\ell_j}=e^{i\eta_j} U^\ve(x;\eta)|_{x_j=-\ell_j}, \quad |x_{3-j}|<\ell_{3-j},
\\
\label{17}
&&\partial_j U^\ve(x;\eta)|_{x_j=\ell_j}=e^{i\eta_j} \partial_j U^\ve(x;\eta)|_{x_j=-\ell_j}, \quad |x_{3-j}|<\ell_{3-j},
\end{eqnarray}
where $j=1,2$, $\partial_j=\partial/\partial x_j$ and $\eta=(\eta_1,\eta_2)$ is the Floquet parameter
in the closed rectangle $R=[0,\pi \ell_1^{-1}]\times [0,\pi\ell_2^{-1}]$. In the sequel we do not 
always display the dependence on $\eta$ explicitly. The problem has  the variational formulation
\begin{equation}
\label{18}
(a^\ve \nabla_x U^\ve, \nabla_x  V)_Q=\Lambda^\ve(\eta) (U^\ve, V)_Q \quad \forall V\in H^1_\eta(Q),
\end{equation}
where $H^1_\eta(Q)$ is the Sobolev space of functions satisfying the conditions \eqref{16}.
The  bilinear form on the left of \eqref{18} is positive and closed in
$H^1(Q)$.
Hence, since the embedding $H^1(Q)\subset L^2(Q)$ is compact, the spectrum of the problem \eqref{18}
or \eqref{15}-\eqref{17} is discrete and turns into the monotone unbounded sequence
\begin{equation}
\label{19}
0\leq\Lambda_1^\ve(\eta)\leq \Lambda_2^\ve(\eta)\leq \ldots \leq\Lambda_k^\ve(\eta)\leq\ldots\to +\infty ,
\end{equation}
and  the corresponding eigenfunctions $U_{1}^\ve(\cdot;\eta)$, $U_{2}^\ve(\cdot;\eta)$, \ldots
can be subject to the normalization and orthogonality conditions
\begin{equation}
\label{20}
(U_{j}^\ve, U_{k}^\ve)_Q=\delta_{j,k}, \quad j,k\in\bbN,
\end{equation}
where $\delta_{j,k}$ is the Kronecker symbol.

The functions ${R}\ni \eta \mapsto \Lambda_k^\ve(\eta)$ are continuous and $\pi \ell_j^{-1}$-periodic
in the variable $\eta_j$, cf. \cite[Ch.VII]{Kato}. As was verified for example in
\cite{KuchUMN,Skrig,Kuchbook},
the spectrum of the problem \eqref{6} or \eqref{8} has band-gap structure,
\begin{equation}
\label{21}
\sigma(\cT^0(\ve))=\bigcup_{k\in\bbN}\beta_k^\ve, \quad \beta_k^\ve=
\big\{\Lambda_k^\ve(\eta)|\eta\in R \big\},
\end{equation}
where the sets $\beta_k^\ve$ are closed finite intervals. Our actual objective is
to describe the sets in \eqref{21} asymptotically as $\ve\to +0$.

\subsection{Limit model problem and theorem on asymptotics.}
\label{sec2.2}
We will next study the relation of the  eigenvalues  \eqref{19} and the spectrum of the so-called 
limit problem
\begin{eqnarray}
\label{B1}
&&-\Delta_x w(x)=\mu w(x),\quad x\in Q,
\\
\label{B2}
&&\partial_n w(x)=0,\quad x\in\partial Q,
\end{eqnarray}
where
$\Delta_x$ is the Laplace operator in the variables $x$ and $\partial_n$ is the outward normal derivative.
The problem \eqref{B1}, \eqref{B2} can be solved explicitly.
Its spectrum consists of the eigenvalue sequence
$\{\mu_n\}_{n\in\bbN}=\big\{\frac{\pi^2}{4}(j^2\ell_1^{-2}+k^2\ell_2^{-2})\big\}_{j,k\in
\bbN\cup\{0\}}$, which is  indexed taking into account  multiplicities such that 
\begin{equation}
\label{B5}
0=\mu_1<\mu_2\leq\mu_{3}\leq\ldots\leq\mu_n\leq\ldots\to+\infty .
\end{equation}
To simplify forthcoming calculations  we assume that $\ell_1^2 \ell_2^{-2}$ is not rational. This guarantees  that
all eigenvalues in \eqref{B5} are simple.
The corresponding eigenfunctions
\begin{equation}
\label{wn}
w_{n}(x)=c_{jk}\cos (\pi (2\ell_1)^{-1} j (x_1+\ell_1))\cos (\pi (2\ell_2)^{-1}(x_2+\ell_2)),
\end{equation}
with $c_{jk}^2 =(1+\delta_{j,0})(1+\delta_{k,0})(\ell_1\ell_2)^{-1}$ satisfy the normalization and orthogonality conditions
$(w_{n},w_{m})_Q=\delta_{n,m}$, $n,m\in\bbN$.

We note that the problem \eqref{B1}-\eqref{B2} has  the variational form
\begin{equation}
\label{B4}
(\nabla_x w, \nabla_x v)_Q=\mu (w,v)_Q \quad \forall v\in H^1(Q)\,.
\end{equation}

The main result in Section \ref{sec2} is the following assertion, the proof of which will be
completed in Section \ref{sec2.4}.

\begin{theorem}\label{th1}
For every $n\in \bbN$, there exist positive $\ve_n$ and $c_n$ such that the 
eigenvalues \eqref{19} and \eqref{B5} are related by
\begin{equation}
\label{BB}
|\Lambda_n^\ve(\eta)-\mu_n|\leq c_n \ve^{\gamma-1/2} \quad \mbox{for} \quad \ve\in (0, \ve_n].
\end{equation}
\end{theorem}

\subsection{Convergence theorem and identification of spectral gaps.}
\label{sec2.3}
Let us denote by $\mu_n^D$  the $n$th eigenvalue (ordered as in \eqref{B5})
for the Dirichlet problem in $Q$, consisting of
the differential equation \eqref{B1} and the  boundary condition $w=0$ on $\partial Q$
instead of \eqref{B2}.
By the max-min principle, see e.g., \cite[Thm. 10.2.2]{BiSo}, \cite[Thm. XIII 1,2]{RS2} we readily
conclude that
$\Lambda_n^\ve(\eta)\leq \mu_n^D$.
Then, for the eigenfunction $U_n^\ve$ of the problem \eqref{15}-\eqref{17}, we have
\begin{eqnarray}
\label{F1}
& & \|\nabla_x U_n^\ve;L^2(Q_{2\ve})\|^2+\|\sqrt{a^\ve}\nabla_x U_n^\ve;L^2(Q_{\ve}
\setminus Q_{2\ve} )\|^2
\nonumber \\
& & +  \ve^{2\gamma}\|\nabla_x U_n^\ve;L^2(Q\setminus Q_\ve)\|^2\leq \mu_n^D.
\end{eqnarray}
Denoting the coordinate dilation by $A_{\ve}x =((1-2\ve\ell_{1}^{-1})x_1,(1-2\ve \ell_2^{-1})x_2)$,
the $H^1(Q)$-norm of the function
\begin{equation}
\label{UUn}
{\bf U}_n^\ve(x; \eta)=U_n^\ve(A_\ve x;\eta)
\end{equation}
is uniformly bounded with respect to  $\ve\in (0,1]$ and $\eta\in R$.
Hence, for some  positive  sequence $\{\ve_p\}_{p\in\bbN}$ converging to 0, we have

\begin{equation}
\label{FL}
\Lambda_n^{\ve_p} (\eta) \to \Lambda_n^0(\eta), \quad
{\bf U}_n^{\ve_p} \rightharpoondown {\bf U}_n^0 \quad \mbox{as } p \to \infty,
\end{equation}
where the latter convergence happens weakly in $H^1(Q)$ and strongly in $L^2(Q)$.

Let   $v^0$ be an arbitrary smooth function in $\overline{Q}$ and set
\begin{equation}
\label{F3}
v^\ve(x)=X^\ve(x) v^0(A_\ve^{-1}x),
\end{equation}
where $X^{\ve} : Q \to [0,1]$ is a smooth cut-off function such that
\begin{equation}
\label{FX}
X^\ve=1 \mbox{ in } Q_\ve,\quad X^\ve=0 \mbox{ in } Q\setminus Q_{\ve/2}, \
\mbox{ and } \ |\nabla_x X^\ve|\leq C_X \ve^{-1} \mbox{ in } Q.
\end{equation}
Since $X^\ve=0$ near $\partial Q$, the function \eqref{F3}
satisfies the quasiperiodicity conditions \eqref{16} and therefore can be inserted into
the integral identity \eqref{18}:
\begin{equation}
\label{F4}
(a^\ve \nabla_x U_n^\ve, \nabla_x v^\ve)_Q=\Lambda_n^\ve(\eta) (U_n^\ve, v^\ve)_Q.
\end{equation}
Here we have
\begin{eqnarray}
\Lambda_n^\ve(\eta) (U_n^\ve, v^\ve)_Q \to \Lambda_n^0(\eta) ({\bf U}_n^0, v^0)_Q
\quad \mbox{as } \ve \to 0,
\label{F4a}
\end{eqnarray}
because, first,
\begin{eqnarray}
\label{F5}
& & (U_n^\ve, v^\ve)_{Q_{2\ve}}=\int_{Q_{2\ve}}{\bf U}_n^\ve(A_\ve^{-1}x)
\overline{v^0(A_\ve^{-1}x)}dx
\nonumber \\
&=&
(1-2\ve \ell_1^{-1})(1-2\ve\ell_2^{-1})({\bf U}_n^\ve,v^0)_Q \to ({\bf U}_n^0, v^0)_Q
\end{eqnarray}
and, second,
$$
\big|(U_n^\ve, v^\ve)_{Q\setminus Q_{2\ve}}\big|\leq c(v^0) \|U_n^\ve; L^2(Q)\| \, |Q\setminus Q_{2\ve}|^{1/2}\leq c_n(v^0)\sqrt{\ve} ,
$$
where we take into account  the normalization condition \eqref{20}, the boundedness of the
function $v^0$ and the area $|Q\setminus Q_{2\ve}|=O(\ve)$ of the integration domain $q\setminus Q_{2\ve}$.
A transformation similar to \eqref{F5}  shows that
\begin{eqnarray}
\label{F6}
(a^\ve \nabla_x U_n^\ve, \nabla_x v^\ve)_{Q_{2\ve}} \to (\nabla_x {\bf U}_n^0,\nabla_x v^0)_Q
\end{eqnarray}
because $a^\ve=1$ on $Q_{2\ve}$. Moreover,
\begin{eqnarray}
& & (a^\ve \nabla_x U_n^\ve, \nabla_x v^\ve)_{Q_{\ve}\setminus Q_{2\ve}}=
(a^\ve \nabla_x U_n^\ve, \nabla_x (v^0\circ A_\ve^{-1}))_{Q_{\ve}\setminus Q_{2\ve}}
\nonumber \\
& \leq &
\|\sqrt{a^\ve} \nabla_x U_n^\ve;L^2(Q_{\ve}\setminus Q_{2\ve})\|\,
\|\sqrt{a^\ve}\nabla_x (v^0\circ A_\ve^{-1});L^2(Q_{\ve}\setminus Q_{2\ve})\|\\
&\leq &\mu_n^D  c_n(v) \ve^{1/2}\,.
\nonumber
\end{eqnarray}
Finally,
\begin{eqnarray}
(a^\ve \nabla_x U_n^\ve, \nabla_x v^\ve)_{Q\setminus Q_\ve}&\leq&
\ve^{2\gamma}\|\nabla_x U_n^\ve;L^2(Q)\|\, \|\nabla_x v^\ve; L^2(Q\setminus Q_\ve)\|
\nonumber \\
& \leq &
\sqrt{\mu_n^D} \ve^{\gamma} C_X\ve^{-1} c_v \ve^{1/2}\leq C_n(v) \ve^{\gamma-1/2}\,.
\label{F7}
\end{eqnarray}
Here, we have used \eqref{F1} to estimate the norm of $\nabla_x U_n^\ve$ and \eqref{F3}, \eqref{FX}
for $\nabla_x v^\ve$. Since $\gamma>1/2$, formulas \eqref{F6}--\eqref{F7}
imply
\begin{eqnarray}
\label{F8}
(a^\ve \nabla_x U_n^\ve, \nabla_x v^\ve)_{Q} \to (\nabla_x {\bf U}_n^0,\nabla_x v^0)_Q
\quad \mbox{as } \ve \to 0.
\end{eqnarray}

We formulate the following  result of our calculations.

\begin{proposition}
For every $n\in \bbN$, the limit $\lambda_n^0(\eta)$  in \eqref{FL} is an eigenvalue 
of the Neumann problem \eqref{B1}, \eqref{B2}, and $\bfU_n^0$ in \eqref{FL} is the
corresponding eigenfunction   with normalization $\|{\bf U}_n^0;L^2(Q)\|=1$.
\end{proposition}

{\bf Proof} The fact that $(\lambda_n^0(\eta), \bfU_n^0)$ is the claimed eigenpair follows from the variational
formulation \eqref{B4}, the arbitrariness of the  choice of $v^0$, the density of smooth functions in 
the Sobolev space, the property \eqref{F4}, and the proven convergence in \eqref{F4a}, \eqref{F8}.

It suffices to verify the normalization of  $\bfU_n^0$.
To this end, we use the inequality
\begin{eqnarray}
\|U_n^\ve;L^2(Q\setminus Q_{2\ve})\|^2\leq c (\ve^2 \|\nabla_x U_n^\ve;L^2(Q\setminus Q_{2\ve})\|^2+\ve \|U_n^\ve; L^2(\partial Q_{2\ve})\|^2)  \label{F10}
\end{eqnarray}
which can be derived by covering the thin frame $Q\setminus Q_{2\ve}$ with sets of diameter $O(\ve)$,
see Fig.\,\ref{fig3}.a),b),
stretching local coordinate systems by a factor of magnitude $\ve^{-1}$ and applying  standard trace
inequalities
in two kinds of sets, see Fig.\,\ref{fig3}.c). For the right hand side of \eqref{F10} we use the 
inequalities 
\begin{eqnarray}
\|U_n^\ve; L^2(\partial Q_{2\ve})\|^2
& \leq &  C (\|\nabla_x U_n^\ve;L^2(Q_{2\ve})\|^2+\|U_n^\ve; L^2(Q_{2\ve})\|^2)
\nonumber \\
& \leq &
C(\Lambda_n^\ve(\eta)+1)\|U_n^{\ve};L^2(Q)\|^2\leq C_n,
\nonumber \\
 \|\nabla_x U_n^\ve; L^2(Q\setminus Q_{2\ve})\|^2
& \leq& \ve^{-2 \gamma} \|\sqrt{a^\ve}\nabla_x U_n^\ve; L^2(Q)\|^2\leq C_n \ve^{-2 \gamma} , \nonumber
\end{eqnarray}
which are based on the estimate \eqref{F1} and the definition of $a^\ve$. As a consequence of
\eqref{20}, \eqref{F10}  and $\gamma < 1$ we get the desired normalization
$$
1=\|U_n^\ve; L^2(Q)\|^2=\|U_n^\ve;L^2(Q_{2\ve})\|^2+O(\ve^{2(1- \gamma) } + \ve)\to \|{\bf U}_n^0;L^2(Q)\|^2. 
$$

\begin{figure}[t]
\begin{center}
\includegraphics[scale=0.8]{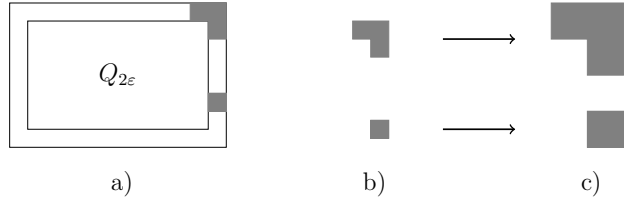}
\end{center}
\caption{The division of the thin frame.}
\label{fig3}
\end{figure}

\bigskip

Formula \eqref{UUn} passes the strong convergence in $L^2(Q)$, see \eqref{FL}, from $\bfU_n^\ve$
to the eigenfunction $U_n^\ve$ itself.

Since the limits of the eigenvalues in \eqref{19} belong to the set $\{\mu_n\}_{n\in\bbN}$
of isolated points, one finds any prescribed
number of open gaps in the spectrum \eqref{21} by assuming the parameter $\ve$ to be sufficiently small (a similar conclusion on the number of spectral bands is made in \cite{Zhik} for the problem introduced in \cite{Hempel}, and the 
same conclusion can be made in \cite{na461}, too).

\subsection{Asymptotics and estimates for spectral bands.}
\label{sec2.4}
In the Hilbert space $\cH^\ve = H^1_\eta (Q)$ we introduce
the scalar product
\begin{equation}
\label{N1}
\langle u,v \rangle_\ve=(a^\ve \nabla_x u, \nabla_x v)_{Q}+(u,v)_{Q}
\end{equation}
and the positive, symmetric, continuous (consequently, self-adjoint) operator $\cK^\ve$,
\begin{equation}
\label{N2}
\langle \cK^\ve u, v\rangle_\ve=(u,v)_Q \quad \forall u,v\in \cH^\ve.
\end{equation}
Comparing \eqref{N1}, \eqref{N2} with \eqref{18}, we see that the variational formulation
of the problem \eqref{15}-\eqref{17} is equivalent to
the abstract equation
\begin{equation*}
\cK^\ve u^\ve=\kappa^\ve u^\ve \mbox{ in } \cH^\ve
\end{equation*}
with the new spectral parameter
\begin{equation}
\label{N4}
\kappa^\ve=(1+\Lambda^\ve)^{-1}\,.
\end{equation}

The well-known formula
\begin{equation}
\label{KK}
\dist (k^\ve, \sigma(\cK^\ve))=\|(\cK^\ve-k^\ve)^{-1}; \cH^\ve\to \cH^\ve\|^{-1},
\end{equation}
follows from the spectral decomposition of the resolvent $(\cK^\ve-k^\ve)^{-1}$,
e.g, \cite[Ch 6, \S3]{BiSo},  \cite[Thm. 12.23]{Rud}.
To estimate the operator norm of the resolvent at the ``interesting'' point $k^\ve=(1+\mu_n)^{-1}$,
we set $\cW^\ve=\|w^\ve; \cH^\ve\|^{-1}w^\ve$, where
$w^\ve(x)=X^\ve(x)w_n(A_\ve^{-1}x)$, $\mu_n$ is an eigenvalue of the problem \eqref{B1}, \eqref{B2},
and the corresponding eigenfunction $w_n$ is extended to the exterior of $Q$ by its formula \eqref{wn}. We have
\begin{eqnarray}
& &
\|\cK^\ve \cW^\ve-k^\ve{\cW}^\ve;\cH^\ve\|= \sup\big|\langle \cK^\ve {\cW}^\ve-
{k}^\ve{\cW}^\ve,v\rangle_\ve\big| \nonumber \\
&  = & (1+\mu_n)^{-1}\| w^\ve; \cH^\ve\|^{-1} \sup\big|(a^\ve \nabla_x w^\ve, \nabla_x
v)_Q-\mu_n(w^\ve, v)_Q \big| ,\label{N6}
\end{eqnarray}
where the supremum is computed over the unit ball in $\cH^\ve$.
The expression inside the modulus  signs in \eqref{N6}
equals the sum of the following terms:
\begin{eqnarray}
&&I_1^\ve=(\nabla_x(w_n\circ A_\ve^{-1}),\nabla_x v)_{Q_{2\ve}}-\mu_n (w_n\circ
A_\ve^{-1},v)_{Q_{2\ve}},\nonumber\\
\label{NII}
&&I_2^\ve=(\sqrt{a^\ve}\nabla_x (w_n^\ve\circ A_\ve^{-1}),\sqrt{a^\ve}\nabla_x v)_{Q_{\ve}\setminus
Q_{2\ve}} \nonumber \\ 
& & \ \ \  - \mu_n (w_n^\ve\circ A_\ve^{-1},v)_{Q_{\ve}\setminus Q_{2\ve}},\\
&&I_3^\ve=\ve^{2\gamma} (\nabla_x (X_\ve w_n\circ A_\ve^{-1}),\nabla_x v)_{Q\setminus Q_{\ve}}-\mu_n
(X_\ve w_n\circ A_\ve^{-1},v)_{Q\setminus Q_{\ve}}.
\nonumber
\end{eqnarray}
Stretching variables and taking \eqref{B4} into account yield
$$
|I_1^\ve|=\big|2\ve\ell_2^{-1} (\partial_{x_1} w_n^{\ve},\partial_{x_1} (v\circ
A_\ve))_Q+2\ve\ell_1^{-1}
(\partial_{x_2} w_n^{\ve},\partial_{x_2} (v\circ A_\ve))_Q\big|\leq c_n \ve\,.
$$
Since $w_n$ is a smooth function, we have
$$
|I_2^\ve|\leq c_n |Q_\ve\setminus Q_{2\ve}| \,  \|v; \cH^\ve\|\leq c_n \ve^{1/2}.
$$
In the same way, taking into account the bound for $\nabla_x X_\ve$ in \eqref{FX},
we obtain
$$
|I_3^\ve|\leq C_n (\ve^{\gamma}\ve^{-1}\ve^{1/2}\|\ve^{\gamma}\nabla_x v; L^2(Q\setminus Q_{\ve})\|+\ve^{1/2} \|v;L^2(Q)\|)\leq c_n \ve^{\gamma-1/2}.
$$

These estimates for the  terms in \eqref{NII} and  \eqref{N6} show that
the norm of the resolvent $(\cK^\ve-k^\ve)^{-1}$ exceeds $c_n \ve^{-(\gamma-1/2)}$ for
some constant $c_n>0$. Thus, in view of the relation \eqref{KK}, the interval 
$$
[k^\ve-c_n\ve^{\gamma-1/2}, k^\ve+c_n \ve^{\gamma-1/2}]
$$
contains an eigenvalue of $\cK^\ve$. Furthermore, the identity  \eqref{N4} shows
that at least one eigenvalue
in \eqref{19} falls into the short segment $\Upsilon_n=[\mu_n-C_n \ve^{\gamma-1/2}, \mu_n+C_n \ve^{\gamma-1/2}]$ with some $C_n>0$ (recall that $\gamma>1/2$) .
To conclude that this eigenvalue is unique and coincides with $\Lambda_n^\ve(\eta)$, we use
Proposition 2.2. If one of the segments $\Upsilon_1, \Upsilon_2,\ldots, \Upsilon_n$ includes two
eigenvalues,
then $\Lambda_{n+1}^\ve(\eta)$ does not exceed $\mu_n+C_n\ve^{\gamma-1/2}$ and, therefore,
converges to $\Lambda_{n+1}^0(\eta)\leq \mu_n$, while the limit $U_{n+1}^0$ of the corresponding
eigenfunction is orthogonal to $w_1$, $w_2$, \ldots, $w_n$ in $L^2(Q)$. Of course this is
impossible because the eigenvalues $\mu_1$, $\mu_2$,\ldots, $\mu_n$ are
simple, due to our assumption on irrationality of $\ell_1^2\ell_2^{-2}$. This completes the
proof of Theorem 2.1.

\section{Asymptotic analysis of the spectrum for composite medium}
\label{sec3}

\subsection{Problem with periodic coefficients in half-planes}
\label{sec3.1}
Let us define the new coefficient function
\begin{equation}
\label{L1}
\bfa^\ve(x)=
\left\{
\begin{array}{l}
a^{\ve} (x_1-h, x_2), \quad x_1>0,\\
a^\ve (x_1+h, x_2), \quad x_1<0,
\end{array}
\right.
\end{equation}
where $h\in (0,\ell_1)$ and the numbers $\ell_j$ are rescaled as $\ell_2=1/2$ and  $\ell_1 >1/2$. Here, we realize the reflection on Fig. \ref{fig2},b.
The  geometric setting is simple enough so that the function  \eqref{L1} can be made smooth
by a proper choice of the old one \eqref{7} inside the thin frame $Q_{2\ve}\setminus Q_\ve$ (for example, $a^\varepsilon$ is independent of $x_1\in (-\ell_1 +3\varepsilon, \ell_1 +3\varepsilon$)). The
difference between  \eqref{6} and the new  equation
\begin{equation}
\label{L2}
-\div (\bfa^\ve (x) \nabla_x \bfu^\ve (x))=\bfla^\ve \bfu^\ve(x),\quad x\in \bbR^2,
\end{equation}
is the loss of the periodicity in the $x_1$-direction due to the coefficient \eqref{L1}:
as indicated in Fig.\,\ref{fig2}.b), the two half-planes, which are 
paved with identical rectangles of size $2\ell_1\times 2\ell_2$, are now separated by a column
of rectangles of size $2(\ell_1+h)\times 2\ell_2$.

Let us denote by $\cT(\ve)$ the self-adjoint operator of the problem \eqref{L2},
defined in the same way as in Section \ref{sec1.2}.

\subsection{Partial FBG-transform and model problem in the unit strip.}
\label{sec3.2}
Let us examine the spectrum of the problem \eqref{L2}. To this end, we apply
the partial FBG-transform
\begin{equation*}
\bfu^\ve(x)\mapsto \bfU^\ve(x;\zeta)=\frac{1}{\sqrt{2\pi}}\sum_{k\in\bbZ} e^{-i\zeta k} \bfu^\ve(x_1,x_2+k), \quad \zeta\in[0,2\pi]
\end{equation*}
and arrive at the model problem in the horizontal unit strip $\Pi=\bbR \times (-1/2, 1/2)$
\begin{equation}
\label{L4}
\begin{array}{l}
-\div (\bfa^\ve(x)\nabla_x \bfU^\ve(x;\zeta))=\bfLa^\ve \bfU^\ve(x,\zeta), \quad x\in \Pi,
\\
\bfU^{\ve}\big(x_1,\frac{1}{2};\zeta\big)=e^{i\zeta} \bfU^{\ve}\big(x_1,-\frac{1}{2};\zeta\big), \quad x_1\in\bbR,
\\
\partial_{x_2}\bfU^{\ve}\big(x_1, \frac{1}{2};\zeta\big)=e^{i\zeta} \partial_{x_2}\bfU^{\ve}\big(x_1,-\frac{1}{2};\zeta\big), \quad x_1\in\bbR.
\end{array}
\end{equation}

It is known, see \cite[Thm. 5]{na17}, that for any fixed $\zeta$
the essential spectrum of the problem \eqref{L4} is the union of the spectral bands
$$
\bfB_k^\ve(\zeta)=\{\Lambda_k^\ve(\eta_1,\zeta)\, : \,
\eta_1\in[-\pi,\pi]\}
, \quad k\in \bbN.
$$
Moreover, there holds the relations
$
\bfB_k^\ve(\zeta) \subset \beta_k^\ve$, $k\in \bbN .
$

The variational formulation of the problem \eqref{L4} is
$$
(\bfa^\ve\nabla_x \bfU^\ve,\nabla_x \bfV)_\Pi= \bfLa^\ve(\zeta)(\bfU^\ve,\bfV)_\Pi \quad \forall \,
\bfV\in \bfH^1_\zeta(\Pi) ,
$$
where $\bfH^1_\zeta(\Pi)$ is the space of functions in $H^1(\Pi)$
satisfying the first quasiperiodicity condition in \eqref{L4}.

\subsection{Asymptotics of eigenvalues and trapped modes in the strip.}
\label{sec3.3}
The appearance of the longer rectangle $\bfQ=\bfQ_1=(-\ell_1-h,\ell_1+h)\times (-1/2,1/2)$ in the
paving of $\Pi$ leads to the new limit problem
\begin{equation}
\label{L5}
-\Delta_x \bfw(x)=\bfmu \bfw(x) \mbox{ in }  \bfQ \, , \quad \partial_\nu \bfw(x)=0
\mbox{ in }   \partial\bfQ\,.
\end{equation}
The first positive eigenvalue of this problem is  $\bfmu_2=\frac{\pi^2}{4}(\ell_1+h)^{-2}$,
corresponding to the eigenfunction $\bfw_2(x)=(\ell_1+h)^{-1/2} \sin (\pi(\ell_1+h)^{-1} x_1/2)$.
Notice that  \eqref{B5} implies
\begin{equation}
\label{L6}
\bfmu_2\in (\mu_1,\mu_2)\,.
\end{equation}

\begin{theorem}
\label{ThL}
For any $\zeta\in[-\pi,\pi]$ there exist positive $\bfve_2$ and $\bfc_2$ such that the problem
\eqref{L4} has an eigenvalue $\bfLa_2^\ve(\zeta)$ satisfying the inequality
\begin{equation}
\label{L7}
|\bfLa_2^\ve(\zeta)-\bfmu_2|\leq \bfc_2 \ve^{\gamma-1/2}\quad \forall \, \ve\in(0,\bfve_2).
\end{equation}
\end{theorem}
\begin{proof}
We set $\bfW(x)=\bfX^\ve(x)\bfw(\bfA_\ve^{-1}x)$, where
$\bfX^{\ve}$ is a smooth cut-off function such that
\begin{equation*}
\bfX^\ve=1 \mbox{ in } \bfQ_\ve,\quad \bfX^\ve=0 \mbox{ in } \bfQ\setminus \bfQ_{\ve/2},
\quad  |\nabla_x \bfX^\ve|\leq C_\bfX \ve^{-1}  \mbox{ in } \bfQ ,
\end{equation*}
and $\bfA_{\ve}x =((1-2\ve(\ell_{1}+h)^{-1})x_1,(1-2\ve \ell_2^{-1})x_2)$. It is enough to estimate
\begin{equation*}
\sup\big|(\bfa^\ve \nabla_x \bfw, \nabla_x \bfv)_\Pi-\bfmu (\bfW, \bfv)_\Pi \big|=\sup\big|(\bfa^\ve \nabla_x \bfW, \nabla_x \bfv)_{\bfQ}-\bfmu
(\bfW, \bfv)_{\bfQ} \big|
\end{equation*}
where the supremum is computed over  the unit ball of the Hilbert space $\bfH^1_\zeta(\Pi)$
with the scalar product
$$
\langle \bfu,\bfv \rangle_{\Pi,\ve}= (\bfa^\ve \nabla_x \bfu,\nabla_x \bfv)_{\Pi}+(\bfu,\bfv)_{\Pi}\,.
$$
This can be done repeating word by word our arguments in the second part of the proof of Theorem 2.1 in Section \ref{sec2.4}.
\end{proof}

Comparing formulas \eqref{L7}, \eqref{L6} and \eqref{BB}, we see that
if $\ve$ is small enough, the spectrum  $\sigma(\cT(\ve))$ of the problem \eqref{L2} contains, in addition to the spectral
bands $\beta_k^\ve$ of the spectrum $\sigma(\cT^0(\ve))$, at least one spectral band
\begin{equation}
\label{Bbold}
\textbf{B}_2^\ve=\{\Lambda_2^\ve(\zeta) \, : \,  \zeta\in[-\pi,\pi]\}
\end{equation}
which does not intersect the set $\sigma(\cT^0(\ve))$. This observation gives a
negative answer to the question \eqref{5} in Section \ref{sec1.1}.

\begin{theorem}\label{thmain}
There exists positive $\bfve_0$ such that, for any $\varepsilon\in (0,\bfve_0)$, the spectrum $\sigma_{\rm ess}(\cT(\ve))$ of the problem \eqref{L2} contains the spectral band \eqref{Bbold} which does not intersect the spectrum $\sigma_{\rm ess}(\cT^0(\ve))$ of the problem \eqref{6}.
\end{theorem}

It is quite obvious that using the techniques presented above one
could prove more  comprehensive results than Theorem \ref{ThL}. Indeed,
many of the open spectral gaps between bands $\beta_k^\ve$ apparently contain
eigenvalues of the limit problem \eqref{L5}. Each of these isolated eigenvalues
gives rise again for a small $\varepsilon$ to an eigenvalue of the problem \eqref{L4} and thus also to an additional
spectral band of the problem \eqref{L2}. However, for the sake of the shortness
of the paper we refrain from going into the detailed proofs, although we are
convinced that a more complete asymptotic description of the eigenvalues of the
problem \eqref{L4} would not require new ideas in addition to those given above.

\section{Concluding remarks}
\label{sec4}
The existence of Rayleigh waves \cite{Ray} travelling along interfaces in piecewise
homogeneous elastic solids is well-known, cf. \cite{KamKisel}, \cite{SO} and others.
Such waves do not exist in the case of scalar differential equations, the piecewise constant
coefficients  of which have jumps at a straight line of the plane. However, the example constructed
above shows  that scalar second order equations with periodic coefficients may have
propagating waves localized near infinite rows and columns of foreign inclusions.
This was already predicted in \cite{CaNaTa}.

Of course, changing the roles of coordinate axis as indicated in Fig.\,\ref{fig2}.c)
provides a row of bigger rectangles $\bfQ_2$ and also new spectral bands in the same way as in
Section 3.  Moreover, according to \cite{CaNaTa}, these bands are preserved in the spectra, if the open
waveguides containing a full  row or column of rectangles are replaced by the corresponding
semi-infinite open waveguides. Combining both of these constructions, we can
create ${\sf X}$-, ${\sf T}$- and ${\sf Y}$-shaped waveguides, which support propagating localized waves, cf. \cite{CaNaTa}.
We also mention the papers \cite{CaDuNa, Na, CaKh} with other examples of localized propagating waves.

Let us consider the ${\sf X}$-shaped open waveguide in  Fig.\,\ref{fig2}.d), which
contains the rectangle $\bfQ_{12}=\{x: |x_1|<\ell_1+h_1, |x_2|<\ell_2+h_2\}$.
The numbers $h_j\in (0,\ell_j)$ can be chosen such that the smallest positive eigenvalues of the Neumann 
problems \eqref{L5} in $\bfQ_j$, $j=1,2$, and that of the
problem \eqref{B1}, \eqref{B2} in $Q$ can be ordered as follows:
\begin{equation}
\label{NEW}
0<\frac{\pi^2}{(\ell_1+h_1)^2}<\frac{\pi^2}{(\ell_2+h_2)^2}<\frac{\pi^2}{\ell_1^2}\,.
\end{equation}
In addition, if $\ell_1<\sqrt{3}(\ell_2+h_2)$, then   $h_j$ can still be adjusted
to obtain
\begin{equation}
\label{mubold}
\frac{\pi^2}{(\ell_2+h_2)^2}<\bfmu_{1,2}=\frac{\pi^2}{(\ell_1+h_1)^2}+\frac{\pi^2}{(\ell_2+h_2)^2}<\frac{\pi^2}{\ell_1^2} ,
\end{equation}
where  $\bfmu_{1,2}$ is the Neumann eigenvalue for $-\Delta$ in $\bfQ_{12}$   corresponding
to the  eigenfunction
$$
\sin\big(\frac{\pi x_1}{\ell_1+h_1}\big)\sin\big(\frac{\pi x_2}{\ell_2+h_2}\big) .
$$
Now consider the spectral problem \eqref{L2}, where the coefficient $\bfa^\ve$
is related to the ${\sf X}$-shaped open waveguide  of Fig.\,\ref{fig2}.d).
According to \cite{CaNaTa} and the conclusions in Sections 2 and 3, the first four spectral bands
of this problem lie in the $c\ve^{\gamma-1/2}$-neighbourhood of the points \eqref{NEW}, although
$\bfmu_{1,2}$ is not contained in  these bands. Thus, our previous asymptotic constructions,
estimates and arguments prove that
there exists  
an isolated  eigenvalue in the vicinity of the point $\bfmu_{1,2}$.

\begin{proposition}
Let $\ell_j$ and $h_j$, $j=1,2$, be fixed to fulfil the relations \eqref{NEW} and \eqref{mubold}. There exists $\bfve _d >0$ such
that, for any $\varepsilon\in (0,\bfve_d)$, the discrete spectrum of the problem corresponding to 
the ${\sf X}$-shaped open waveguide in fig. \ref{fig2},d, contains at least one eigenvalue $\lambda_d(\ve)=\bfmu_{1,2} +O(\ve^{\gamma-1/2})$.
\end{proposition}

Recall that if $\ve$ is small, we have shown the existence of many open spectral gaps,
cf. for example the end of Section \ref{sec2.3}. It might be possible to find
also  other eigenvalues (of the problem related to Fig.\,\ref{fig2}.d) inside these gaps,
just by using the  above scheme to locate them  near suitable Neumann eigenvalues  of
the problem in $\bfQ_{12}$. However, the first couple of the positive
eigenvalues $\mu_j=\pi^2 (\ell_j+h_j)^{-2}$, $j=1,2$, coincides with the numbers in \eqref{NEW}
and therefore we do not know if they are included in the corresponding spectral bands or not.
In other words, to prove or disprove the existence of isolated eigenvalues near $\bfmu_j$ one would
need to  construct higher order terms in the asymptotic expansions.

An example of an eigenvalue embedded into the continuous spectrum of an open waveguide in
a double periodic medium does not yet exist in the literature.
We conjecture that this could be done using the concept of  enforced
stability of embedded eigenvalues,  \cite{na489, na546},
although  it would require a  much more delicate asymptotic analysis.


\subsection*{Acknowledgment}

F.B. was supported by grants 0.38.237.2014 and 6.57.61.2016 of St.Petersburg University, and by grant 15-01-02175 
of RFBR; S.N. by grant 15-01-02175 
of RFBR and by the Academy of Finland project 268973;  
G.C. is a member of GNAMPA of INDAM; J.T. was partially supported by the V\"{a}is\"{a}l\"{a} Foundation of 
the Finnish Academy of Sciences and Letters.

\end{document}